\documentclass[a4paper,11pt]{article}
\usepackage{amsfonts,amsmath,amssymb,amsthm,verbatim,placeins,xcolor,comment,placeins,caption,subcaption,enumerate,bm,graphicx}
\usepackage{geometry}
\usepackage[pdfauthor={DG},pdfstartview={FitH},colorlinks=true,citecolor=blue]{hyperref} 
\usepackage[english]{babel}
\usepackage[sort&compress]{natbib}
\usepackage{pgfplots}

\newtheorem{theorem}{Theorem}[section]
\newtheorem{prop}{Proposition}[section]
\newtheorem{lemma}{Lemma}[section]

\theoremstyle{remark}

\theoremstyle{definition}

\theoremstyle{remark}
\newtheoremstyle{myremark}{}{}{\color{blue}\small}{}{\color{blue}\bfseries}{}{ }{}

\theoremstyle{myremark}

\newcommand{\E}{\mathbb{E}} 
\newcommand{\R}{{\mathbb R}}
\newcommand{\N}{{\mathbb N}}

\newcommand\1{\mathbf{1}}

\newcommand{\RV}{\operatorname{RV}}
\newcommand{\hX}{\widehat{X}}
\newcommand{\hPsi}{\widehat{\Psi}}

\usepackage{marginnote}

\begin{document}

\renewcommand*{\thefootnote}{\fnsymbol{footnote}}

\begin{center}
\Large{\textbf{Intermittency in the small-time behavior of L\'evy processes}}\\
\bigskip
\bigskip
Danijel Grahovac$^1$\footnote{dgrahova@mathos.hr}\\
\end{center}

\bigskip
\begin{flushleft}
\footnotesize{
$^1$ Department of Mathematics, University of Osijek, Trg Ljudevita Gaja 6, 31000 Osijek, Croatia}
\end{flushleft}

\bigskip

\textbf{Abstract: } In this paper we consider convergence of moments in the small-time limit theorems for L\'evy processes. We provide precise asymptotics for all the absolute moments of positive order. The convergence of moments in limit theorems holds typically only up to some critical moment order and higher order moments decay at different rate. Such behavior is known as intermittency and has been encountered in some limit theorems. 

\medskip


\bigskip

\section{Introduction}
In classical limit theorems, one typically studies the behavior of some form of aggregated process as time tends to infinity. For example, if $\{\xi_i, \, i \in \N\}$ is an i.i.d.~sequence and $(a_\lambda) \subset \R$, $a_\lambda\to \infty$, then it is well known that the class of weak limits of
\begin{equation}\label{e:intr1}
\left\{ \frac{1}{a_\lambda} \sum_{i=1}^{\lfloor \lambda t \rfloor} \xi_i, \, t \geq 0 \right\}, \quad \text{ as } \lambda\to \infty,
\end{equation}
coincides with the class of self-similar L\'evy processes, namely, Brownian motion or infinite variance strictly stable processes. If $\xi_i$ are infinitely divisible, then one-dimensional marginals of \eqref{e:intr1} may be identified with $a_\lambda^{-1} X(\lfloor \lambda t \rfloor)$ for a L\'evy process $X$ with increments distributed as $\xi_i$. The type of the limiting process depends solely on the tail behavior of the distribution of $\xi_i$. If $\xi_i$ is infinitely divisible, then this behavior is related to the behavior of the L\'evy measure of $\xi_i$ at infinity.

Alternatively, one may consider limit theorems for a small-time limiting scheme (or ``zooming in'' as termed by \cite{ivanovs2018zooming}), where one considers for a L\'evy process $X$ the limit
\begin{equation}\label{e:intr2}
\left\{ \frac{1}{a_\lambda} X (\lambda t ), \, t \geq 0 \right\}, \quad \text{ as } \lambda\to 0.
\end{equation}
Although the class of possible limits is the same as for \eqref{e:intr1}, the domains of attraction now crucially depend on the behavior of the L\'evy measure of $X$ near zero (see  \cite{doney2002stability,maller2008convergence,ivanovs2018zooming}).

In this paper we investigate small-time behavior of absolute moments of L\'evy processes. Assuming that \eqref{e:intr2} converges to some non-trivial process $\hX$, we show that the normalized absolute moments $a_\lambda^{-q} \E |X(\lambda t)|^q$ converge to the absolute moments of the limit, but this typically holds only up to some critical moment order $\alpha$. Absolute moments of the order $q$ greater than $\alpha$ cannot be normalized with $a_\lambda^{q}$. This implies that the $L^q$ norms of $X(\lambda t)$ decay at different rates for different range of $q$. Such a behavior is known as intermittency (see \cite{GLST2016JSP,GLST2019Bernoulli}) and resembles a similar phenomenon appearing in solutions of some stochastic partial differential equations (SPDE) (see e.g.~\cite{carmona1994parabolic,gartner2007geometric,zel1987intermittency,khoshnevisan2014analysis,chong2018almost,chong2019intermittency}). Intermittency in the SPDEs usually involves long-term behavior of solutions corresponding to $\lambda \to \infty$ and in limit theorems one is likewise interested in the limits of aggregated processes as $\lambda \to \infty$ (see e.g.~\cite{GLT2019Limit,GLT2019LimitInfVar,GLT2019MomInfVar,grahovac2018intermittency}). Here, however, we find an instance of intermittent behavior appearing in the small-time limit.

We note that there exists a large body of literature related to the small-time behavior of L\'evy processes: \cite{bertoin2008passage,maller2009small,maller2008convergence,doney2002stability,aurzada2013small,maller2015strong,ivanovs2018zooming}. Small-time behavior of moments of L\'evy processes has been investigated in \cite{figueroa2008small} (see also \cite{woerner2003variational,jacod2007asymptotic,asmussen2001approximations}). However, these results apply only to moments of the order greater than the Blumenthal-Getoor index of the process. We provide here the asymptotics for the full range of positive order moments provided that \eqref{e:intr2} converges to a non-trivial limit. In \cite{deng2015shift}, bounds are established for absolute moments of general L\'evy processes applicable both for small and large time.

The paper is organized as follows. We start with some preliminaries in Section \ref{sec2} and discuss the notion of intermittency in small-time in Section \ref{sec3}. In Section \ref{sec4} we establish asymptotic behavior of moments. In Section \ref{sec5} we shortly discuss the so-called multifractal formalism which relates behavior of moments with the sample path properties of the process.

\section{Preliminaries}\label{sec2}

Through the paper, $X=\{X(t), \, t \geq 0\}$ will denote a L\'evy process. Let $\Psi$ be its characteristic exponent $\log \E \left[ e^{i \zeta X(t)} \right] = t \Psi (\zeta)$ for $t\geq 0$, $\zeta \in \R$, which by the L\'evy-Khintchine formula has the following form
\begin{equation}\label{LKformula}
\Psi(\zeta) = i \gamma \zeta - \frac{\sigma^2}{2} \zeta^2 + \int_{\R} \left(e^{i\zeta x} - 1 - i \zeta x \mathbf{1}_{\{|x|\leq 1\}} \right) \Pi(dx),
\end{equation}
with $\gamma \in \R$, $\sigma>0$ and $\Pi$ a measure on $\R \backslash \{0\}$ such that $\int (1\wedge x^2) \Pi(dx) < \infty$. We refer to $(\gamma,\sigma,\Pi)$ as the characteristic triplet. The L\'evy process has paths of bounded variation if and only if $\sigma=0$ and $\int_{|x|\leq 1}|x| \Pi(dx)<\infty$. If $\int_{|x|\leq 1}|x| \Pi(dx)<\infty$, then \eqref{LKformula} may be written in the form
\begin{equation}\label{LKformulabv}
\Psi(\zeta) = i \gamma' \zeta - \frac{\sigma^2}{2} \zeta^2 + \int_{\R} \left(e^{i\zeta x} - 1\right) \Pi(dx),
\end{equation}
and $\gamma'$ will be referred to as drift. For more details see \cite{sato1999levy,kyprianou2014fluctuations,bertoin1998levy}.

Let $\{Y(t), \, t\geq 0\}$ be a general real-valued process. It is said to be self-similar with index $H>0$ if for any $c>0$ it holds that $\{Y(ct)\} \overset{d}{=} \{c^H Y(t)\}$, where $\{\cdot\} \overset{d}{=} \{\cdot\}$ denotes the equality of the finite dimensional distributions of two processes. If $X$ is a self-similar L\'evy process, then $X$ is either
\begin{itemize}
\item Brownian motion in which case $H=1/2$ and the characteristic triplet is $(0,\sigma,0)$,
\item linear drift in which case $H=1$ and the characteristic triplet is $(\gamma,0,0)$, $\gamma\neq 0$,
\item strictly $\alpha$-stable L\'evy process with $0<\alpha<2$ in which case $H=1/\alpha$ and the characteristic triplet is $(\gamma,0,\Pi)$ with
\begin{equation*}
\Pi(dx)=c_+ x^{-1-\alpha} \1_{\{x>0\}} dx + c_- |x|^{-1-\alpha} \1_{\{x<0\}} dx,
\end{equation*}
for some $c_+,c_- \geq0$, $c_++c_->0$ and, additionally, $c_+=c_-$ if $\alpha=1$, while $\gamma$ is given by $\gamma = (c_+-c_-)/(1-\alpha)$.
\end{itemize}
It follows from Lamperti's theorem (see \cite{lamperti1962semi}, \cite[Thm.~2.8.5]{pipiras2017long}) that the class of self-similar L\'evy processes coincides with the class of weak limits of $\{\sum_{i=1}^{\lfloor \lambda t \rfloor} \xi_i/a(\lambda), \, t \geq 0\}$ as $\lambda\to \infty$, where $\{\xi_i, \, i \in \N\}$ is an i.i.d.~sequence and $a(\lambda)$ is a nonstochastic function such that $a(\lambda) \to \infty$ as $\lambda \to \infty$. By \cite[Theorem 16.14]{kallenberg2002foundations}, a necessary and sufficient coefficient for such convergence to some L\'evy process $\widehat{X}=\{\widehat{X}(t), \, t \geq 0\}$ is that $\sum_{i=1}^{n} \xi_i/a_n \to^d \widehat{X}(1)$. The characterization of domains of attraction follows from \cite[Thm.~7.35.2]{gnedenko1954limit}, \cite{shimura1990strict} and \cite{ivanovs2018zooming} for the convergence to a linear drift.

In a small-time limiting scheme (or ``zooming in'' as termed by \cite{ivanovs2018zooming}) one considers the limit
\begin{equation}\label{e:limscheme}
\left\{  \frac{Y(ts)}{a(t)}, \, s \geq 0 \right\} \overset{fdd}{\to} \{Z(s), \, s \geq 0\}, \quad \text{ as } t \to 0,
\end{equation}
for some processes $\{Y(s), \, s \geq 0\}$, $\{Z(s), \, s \geq 0\}$, a nonstochastic function $a(t)$ such that $a(t)\to 0$ as $t\to 0$, and with convergence in the sense of convergence of all finite-dimensional distributions. By adapting the proof of Lamperti's theorem, one can show that if \eqref{e:limscheme} holds with $Z$ stochastically right-continuous and non-trivial (not identically $0$), then $Z$ must be self-similar with some index $H>0$ and $a(t)$ is regularly varying at zero with index $H$, which we will denote by $a(t)\in \RV^0(H)$ (see \cite[Thm.~1]{ivanovs2018zooming}).

\section{Small-time intermittency}\label{sec3}
Intermittency typically refers to an unusual asymptotic behavior of moments. In \cite{GLST2019Bernoulli}, intermittency is defined as a change in the rate of growth of the $L^q$ norm of the process as time tends to $\infty$. More precisely, for a general real-valued process $Y=\{Y(t),\, t \geq 0\}$ define the scaling function at point $q \in \R$ as
\begin{equation}\label{deftauInf}
\tau^\infty(q) = \tau_Y^\infty(q) = \lim_{t\to \infty} \frac{\log \E |Y(t)|^q}{\log t}.
\end{equation}
The process $Y$ is then intermittent if 
\begin{equation*}
q \mapsto \frac{\tau_Y^\infty(q)}{q} = \lim_{t\to \infty} \frac{\log  \left\lVert Y(t) \right\rVert_q}{\log t},
\end{equation*}
has points of strict increase, where $\left\lVert Y(t) \right\rVert_q= \left(  \E |Y(t)|^q \right)^{1/q}$, which is the $L^q$ norm if $q\geq 1$. The scaling function \eqref{deftauInf} is tailored for measuring the rate of growth of moments as time tends to $\infty$ of the process whose moments grow roughly as a power function of time. For $H$-self-similar process, for example, $\tau^\infty(q)=Hq$ for $q$ in the range of finite moments.

In the limiting scheme \eqref{e:limscheme}, clearly $Y(t) \to^P 0$ as $t \to 0$. In this setting we want to measure the rate of decay of absolute moments to zero as $t\to 0$ in \eqref{deftauInf}. Hence, we define the scaling function as
\begin{equation}\label{deftau0}
\tau^0(q) = \tau_Y^0(q) = \lim_{t\to 0} \frac{\log \E |Y(t)|^q}{\log (1/t)},
\end{equation}
where we assume the limit exists, possibly equal to $\infty$. If $\E|Y(t)|^q=\infty$ for $t\geq t_0$, then $\tau^0(q)=\infty$. If $Y$ is $H$-self-similar process, then $\tau^0(q)=-Hq$ for $q$ in the range of finite moments.

We divide by $\log(1/t)$ instead of $\log t$ in \eqref{deftau0} for convenience. Namely, we have that $\tau_Y^0(q)=\tau_{Y'}^\infty(q)$, where $\tau_{Y'}^\infty$ is the scaling function \eqref{deftauInf} of the process $Y'(t)=Y(1/t)$. From this fact and \cite[Prop.~2.1]{GLST2016JSP}, we immediately get that $\tau^0$ is convex, hence continuous. Moreover, $q \mapsto \frac{\tau^0(q)}{q}$ is nondecreasing on $\mathcal{D}_{\tau^0}=\{q \in \R : \tau^0(q)<\infty\}$.

Following \cite{GLST2016JSP,GLST2019Bernoulli}, we say that the process $\{Y(t), \, t \geq 0\}$ with the scaling function $\tau^0$ given by \eqref{deftau0} is intermittent (in the small-time limit) if there exist $p, r \in \mathcal{D}_{\tau^0}$ such that
\begin{equation*}
    \frac{\tau^0(p)}{p} < \frac{\tau^0(r)}{r}.
\end{equation*}
The following proposition explains some implications of intermittency related to limit theorems. The proof follows directly from \cite[Thm.~1]{GLST2019Bernoulli} and Lamperti's theorem for limiting scheme \eqref{e:limscheme} \cite[Thm.~1]{ivanovs2018zooming}.

\begin{prop}\label{prop:limitLamp}
Suppose that $Y$ satisfies \eqref{e:limscheme} for some  stochastically right-continuous and non-trivial process $Z$. Then $Z$ is $H$-self-similar and the scaling function \eqref{deftau0} of $Y$ is $\tau_Y^0(q)=-Hq$ for every $q$ such that, as $t\to 0$
\begin{equation}\label{e:limitmom}
a(t)^{-q} \E| Y(ts)|^q  \to \E |Z(s)|^q, \quad \forall s \geq 0.
\end{equation}
\end{prop}

Proposition \ref{prop:limitLamp} shows that for intermittent processes obeying limit theorem in the sense of \eqref{e:limscheme}, the convergence of moments as in \eqref{e:limitmom} must fail to hold for some range of $q$.

\section{Small-time moment asymptotics of L\'evy processes}\label{sec4}
Domains of attraction in the limiting scheme \eqref{e:limscheme} have been characterized in \cite{doney2002stability,maller2008convergence,ivanovs2018zooming}. Suppose $X$ is a L\'evy process with characteristic exponent $\Psi$ such that
\begin{equation}\label{e:LPlimit}
\left\{ \frac{X(t s)}{a(t)}, \, s \geq 0 \right\} \overset{fdd}{\to} \{\hX(s), \, s \geq 0\}, \quad \text{ as } t \to 0.
\end{equation}
Then $\hX$ must be a L\'evy process too with some characteristic exponent $\hPsi$, the convergence extends to convergence in Skorokhod space of c\`adl\`ag functions \cite[Cor.~VII.3.6]{jacod1987limit} and is equivalent to $a(t)^{-1} X(t) \overset{d}{\to} \hX(1)$, as $t \to 0$, which in turn is equivalent to
\begin{equation}\label{e:psiconv}
t \Psi(a(t)^{-1} \zeta) \to \hPsi(\zeta), \quad \text{ as } t \to 0, \ \forall \zeta \in \R.
\end{equation}
Moreover, if $\hX$ is non-trivial, then it is a $H$-self-similar L\'evy process for some $H\in [1/2,2]$ and $a(t) \in \RV^0(H)$. In \cite[Thm.~2]{ivanovs2018zooming}, a complete characterization of domains of attraction is provided in terms of the characteristic triplet $(\gamma, \sigma, \Pi)$ of the L\'evy process $X$. Some sufficient conditions can also be found in \cite[Prop.~2.3]{bisewski2019zooming}.

In the following we denote the characteristic triplet of $X$ by $(\gamma, \sigma, \Pi)$, the characteristic triplet of $\hX$ by $(\widehat{\gamma}, \widehat{\sigma}, \widehat{\Pi})$ and we assume \eqref{e:LPlimit} holds. We also define indices
\begin{align}
\beta^0 &:= \inf \left\{ \beta \geq 0 : \int_{|x|\leq 1} |x|^\beta \Pi(dx) < \infty \right\},\label{beta0}\\
\beta^\infty &:= \sup \left\{ \beta \geq 0 : \int_{|x|> 1} |x|^\beta \Pi(dx) < \infty \right\}.\label{betainf}
\end{align}
The index $\beta^0$ is known as the Blumenthal-Getoor index \cite{blumenthal1961sample} and we must have $\beta^0 \in [0,2]$ since $\Pi$ is the L\'evy measure. On the other hand, $\beta^\infty \in [0, \infty]$ is related to the tail behavior of the distribution of $X(1)$. In particular, $\E |X(1)|^q<\infty$ for $q<\beta^\infty$ and $\E |X(1)|^q=\infty$ for $q>\beta^\infty$.

\begin{lemma}\label{lemma:q<beta}
If $X$ and $\hX$ are L\'evy processes such that \eqref{e:LPlimit} holds, then for $0<q< (1/H) \wedge \beta^\infty$, $\{ a(t)^{-q} |X(ts)|^q \}$ is uniformly integrable for every $s\geq 0$ and hence
\begin{equation*}
a(t)^{-q} \E |X(ts)|^q \to \E |\hX(s)|^q, \quad \text{ as } t \to 0.
\end{equation*}
\end{lemma}

\begin{proof}
It is enough to show that $\{ a(t)^{-q} \E |X(ts)|^q \}$ is bounded for arbitrary $0<q<(1/H) \wedge \beta^\infty$. Let $X^{(t)} (s)=a(t)^{-1} X(ts)$ and let $(\gamma^{(t)}, \sigma^{(t)}, \Pi^{(t)})$ denote the characteristic triplet of the L\'evy process $\{ X^{(t)} (s), \, s \geq 0 \}$. Decompose $X^{(t)}$ into independent factors
\begin{equation}\label{e:proofdecomposition}
X^{(t)}(s) = \gamma^{(t)} s + X_1^{(t)}(s) + X_2^{(t)} (s),
\end{equation}
where $X_1^{(t)}$ is a L\'evy process with triplet $(0,\sigma^{(t)}, \Pi_1^{(t)})$, $\Pi_1^{(t)} (dx) = \1_{\{|x|\leq 1\}} \Pi^{(t)}(dx)$, and $X_2^{(t)}$ is a L\'evy process with triplet $(0,0, \Pi_2^{(t)})$, $\Pi_2^{(t)} (dx) = \1_{\{|x|> 1\}} \Pi^{(t)}(dx)$. By the elementary inequality $|a+b|^r\leq 2^r(|a|^r+|b|^r)$ we have from \eqref{e:proofdecomposition}
\begin{equation*}
\E |X^{(t)} (s)|^q \leq C \left( |\gamma^{(t)} s|^q  + \E |X_1^{(t)} (s)|^q  + \E |X_2^{(t)} (s)|^q \right).
\end{equation*}
We now show each of these terms is bounded. 

By \eqref{e:psiconv} and \cite[Thm.~15.14]{kallenberg2002foundations}, $\gamma^{(t)}$, $\sigma^{(t)}$ and $\int_{|x|\leq 1} x^2 \Pi^{(t)} dx$ have finite limits as $t \to 0$ (see also \cite[Lem.~4.8]{bisewski2019zooming}, \cite[Prop.~4.1]{maller2008convergence}). In particular, $\{\gamma^{(t)}\}$ is bounded. Moreover, $\E |X_1^{(t)} (s)|^q$ is bounded. Indeed, since $q \leq 2$, by Jensen's inequality and \cite[Exa.~25.12]{sato1999levy} we have
\begin{equation*}
\E |X_1^{(t)} (s)|^q = \E \left( |X_1^{(t)} (s)|^2 \right)^{q/2} \leq \left( \E |X_1^{(t)} (s)|^2 \right)^{q/2} = \left( s \left(\sigma^{(t)}\right)^2 + s \int_{|x|\leq 1} x^2 \Pi^{(t)} (dx) \right)^{q/2}.
\end{equation*}
For $\E |X_2^{(t)} (s)|^q$ we proceed similarly as in the proof of \cite[Lem.~4.9]{bisewski2019zooming}. By \cite[Lem.~4.8]{bisewski2019zooming} it holds that
\begin{equation}\label{e:xpPIconv}
\int_{|x|>1} |x|^q \Pi^{(t)}(dx) \to \int_{|x|>1} |x|^q \widehat{\Pi}(dx) < \infty,
\end{equation}
Although the statement of \cite[Lem.~4.8]{bisewski2019zooming} is for $t=1/n$, the same proof is valid for general $t$. Note that $X_2^{(t)}$ is a compound Poisson process with intensity parameter $\lambda^{(t)}=\Pi^{(t)}(1,\infty)+\Pi^{(t)}(-\infty,-1)$ and jump distribution $\1_{\{|x|> 1\}} \Pi^{(t)}(dx)/\lambda^{(t)}$. If $N^{(t)}(s)$ is Poisson random variable with parameter $s\lambda^{(t)}$ and $\xi^{(t)}$ a generic jump, then by conditioning on the number of jumps and using Minkowski inequality we get for $q \geq 1$
\begin{align*}
\E |X_2^{(t)} (s)|^q &\leq \E (N^{(t)}(s))^q \E  |\xi^{(t)}|^q \leq  \E (N^{(t)}(s))^2 \frac{1}{\lambda^{(t)}} \int_{|x|>1} |x|^q \Pi^{(t)}(dx)\\
&=\left( s + \lambda^{(t)} s^2 \right) \int_{|x|>1} |x|^q \Pi^{(t)}(dx),
\end{align*}
and this is bounded since $\lambda^{(t)}$ is bounded and \eqref{e:xpPIconv} holds. For $q<1$ we use the inequality $(a+b)^q\leq a^q+b^q$, $a,b\geq 0$ and get
\begin{equation*}
\E |X_2^{(t)} (s)|^q \leq \E N^{(t)}(s) \E  |\xi^{(t)}|^q =s \int_{|x|>1} |x|^q \Pi^{(t)}(dx),
\end{equation*}
which is bounded by \eqref{e:xpPIconv}.
\end{proof}

Following \cite[Cor.~1]{ivanovs2018zooming}, we have that $\beta^0=1/H$ unless $\sigma>0$ or $X$ is bounded variation with $\gamma'\neq 0$. Lemma \ref{lemma:q<beta} covers the moment asymptotics for $q<\beta^0 \wedge \beta^\infty$. For $\beta^0<q<\beta^\infty$ we have the following lemma which is based on the results of \cite{figueroa2008small} (see also \cite{woerner2003variational}).

\begin{lemma}\label{lemma:q>beta}
Let $X$ be a L\'evy process with characteristic triplet $(\gamma, \sigma, \Pi)$ such that $\beta^0<\beta^\infty$ and let $I=(\beta^0, \beta^\infty)$, where $\beta^0$ and $\beta^\infty$ are given by \eqref{beta0} and \eqref{betainf}.
\begin{enumerate}[(i)]
\item If $\sigma=0$ and $\Pi \equiv 0$, then as $t \to 0$
\begin{equation*}
t^{-q} \E |X(ts)|^q \to s^q |\gamma|^q, \quad \text{ for } q\in I.
\end{equation*}
\item If $\sigma\neq 0$ and $\Pi \equiv 0$, then as $t \to 0$
\begin{equation*}
t^{-\frac{q}{2}} \E |X(ts)|^q \to s^{\frac{q}{2}} \E | \mathcal{N}(0,\sigma^2)|^q, \quad \text{ for } q\in I,
\end{equation*}
where $\E | \mathcal{N}(0,\sigma^2)|^q$ is the $q$-th absolute moment of normal distribution with mean $0$ and variance $\sigma^2$.
\item If $\sigma = 0$ and $\Pi \not\equiv 0$, then as $t \to 0$
\begin{equation}\label{s0Pinot0case}
t^{-1} \E |X(ts)|^q \to s \int_\R |x|^q \Pi(dx), \quad \text{ for } q\in I \cap (1,\infty).
\end{equation}
In case $\beta^0<1$, if $\gamma'=0$ in \eqref{LKformulabv}, then \eqref{s0Pinot0case} holds for $q\leq 1$ also and if $\gamma'\neq 0$, then as $t \to 0$
\begin{equation}\label{s0Pinot0case2}
t^{-q} \E |X(ts)|^q \to s^q |\gamma'|^q, \quad \text{ for } q\in I \cap (0,1).
\end{equation}
\item If $\sigma \neq 0$ and $\Pi \not\equiv 0$, then as $t \to 0$
\begin{equation*}
\begin{cases}
t^{-1} \E |X(ts)|^q \to s \int_\R |x|^q \Pi(dx),& \quad \text{ for } q\in I\cap (2,\infty),\\
t^{-1} \E |X(ts)|^q \to s \sigma^2  + s \int_\R  |x|^2 \Pi(dx),& \quad \text{ for } q=2 \text{ and } q\in I,\\
t^{-\frac{q}{2}} \E |X(ts)|^q \to s^{\frac{q}{2}} \E | \mathcal{N}(0,\sigma^2)|^q,& \quad \text{ for } q \in I\cap (0,2).
\end{cases}
\end{equation*} 
\end{enumerate}
\end{lemma}

\begin{proof}
\begin{enumerate}[(i)]
\item is obvious.

\item If we decompose $X$ as $X(t)=\gamma t + X_1(t)$ where $X_1$ is Brownian motion, then by self-similarity
\begin{equation*}
t^{-\frac{q}{2}} \E |X(ts)|^q = \E |\gamma s t^{-\frac{1}{2}+1} + s^{\frac{1}{2}} X_1(1)|^q \to s^{\frac{q}{2}} \E | \mathcal{N}(0,\sigma^2)|^q.
\end{equation*}

\item That \eqref{s0Pinot0case} holds for $q>1$ and for $q\leq 1$ if $\gamma'=0$, follows from \cite[Thm.~1.1]{figueroa2008small}. If $\gamma'\neq 0$, we decompose $X$ as $X(t)=\gamma' t + X_2(t)$, where $X_2$ is a L\'evy process with triplet $(\gamma-\gamma',0,\Pi)$. Note that the drift term of $X_2$ is $0$ and by the previous case we have that $t^{-1} \E |X_2(ts)|^q \to s \int_\R |x|^q \Pi(dx)$ and hence $t^{-q} \E |X_2(ts)|^q \to 0$ for $q<1$. It follows that $|t^{-1} X_2(ts)|^q \to^P 0$ and $\{|t^{-1} X_2(ts)|^q\}$ is uniformly integrable, but then so is $|\gamma' t + n X_2(t/n)|^q$ which gives \eqref{s0Pinot0case2}.

\item The cases $q>2$ and $q=2$ follow from \cite[Thm.~1.1]{figueroa2008small}. For $q<2$ we decompose $X$ into independent components as $X(t)=X_1(t)+X_2(t)$ where $X_1$ is Brownian motion and $X_2$ is L\'evy process with triplet $(\gamma,0,\Pi)$. We have that $t^{-\frac{1}{2}} X_1(ts) = s^{\frac{1}{2}} X_1(1)$ and $t^{-\frac{q}{2}} \E |X_2(ts)|^q \to 0$, by self-similarity and case (iii), respectively. Hence, $|t^{-\frac{1}{2}} X_2(ts)|^q \to^P 0$ and $\{|t^{-\frac{1}{2}} X_2(ts)|^q\}$ is uniformly integrable and so is $|t^{-\frac{1}{2}}X_1(ts) + t^{-\frac{1}{2}} X_2(ts)|^q$ which completes the proof.
\end{enumerate}
\end{proof}

By combining the results of Lemmas \ref{lemma:q<beta} and \ref{lemma:q>beta} we get the scaling function of the L\'evy process satisfying \eqref{e:LPlimit}.

\begin{theorem}\label{thm:tau0}
Suppose that $X$ and $\hX$ are L\'evy processes such that \eqref{e:LPlimit} holds, let  $(\gamma,\sigma,\Pi)$ denote the characteristic triplet of $X$ and  $\tau^0$ its scaling function \eqref{deftau0}.
\begin{enumerate}[(i)]
\item Suppose that $\hX$ is Brownian motion. If $\Pi\not \equiv 0$, then for $q \in (0,\beta^\infty)$
\begin{equation*}
\tau^0(q)=\begin{cases}
-\frac{1}{2} q, & \ \text{ if } 0<q\leq 2,\\
-1, & \ \text{ if } q>2.
\end{cases}
\end{equation*}
If $\Pi \equiv 0$, then $\tau^0(q)=-q/2$ for $q \in (0,\beta^\infty)$.

\item Suppose that $\hX$ is linear drift. If $\Pi\not \equiv 0$, then for $q \in (0,\beta^\infty)$
\begin{equation*}
\tau^0(q)=\begin{cases}
-q, & \ \text{ if } 0<q\leq 1,\\
-1, & \ \text{ if } q>1.
\end{cases}
\end{equation*}
If $\Pi \equiv 0$, then $\tau^0(q)=-q$ for $q \in (0,\beta^\infty)$.

\item If $\hX$ is strictly stable L\'evy process, then for $q \in (0,\beta^\infty)$
\begin{equation*}
\tau^0(q)=\begin{cases}
-\frac{1}{\beta^0} q, & \ \text{ if } 0<q\leq \beta^0,\\
-1, & \ \text{ if } q > \beta^0.
\end{cases}
\end{equation*}
\end{enumerate}
\end{theorem}

\begin{proof}
From Proposition \ref{prop:limitLamp} (i) we have that $\tau^0(q)=-Hq$ for $0<q< (1/H) \wedge \beta^\infty$.
\begin{enumerate}[(i)]
\item By Lemma \ref{lemma:q<beta},  $\tau^0(q)=-q/2$ for $0<q<2$. If $\Pi\not \equiv 0$, then by Lemma \ref{lemma:q>beta} (iii) and (iv) we have $\tau^0(q)=-1$ for $q\geq 2$. If $\Pi \equiv 0$, then we must have $\sigma\neq 0$ by \cite[Thm.~2]{ivanovs2018zooming}. From Lemma \ref{lemma:q>beta} (ii) then $\tau^0(q)=-q/2$ for $q\geq 2$.

\item By Lemma \ref{lemma:q<beta},  $\tau^0(q)=-q$ for $0<q<1$. By \cite[Thm.~2]{ivanovs2018zooming} we must have $\sigma=0$. For $q>1$ we use Lemma \ref{lemma:q>beta} (iii) if $\Pi\not \equiv 0$ and Lemma \ref{lemma:q>beta} (i) if $\Pi\equiv 0$. That $\tau^0(q)=-1$ follows by continuity of $\tau^0$ since $\tau^0$ is convex.

\item In this case we must have $\sigma=0$, $\Pi\not \equiv 0$ and in the bounded variation case $\gamma'=0$ (see \cite[Thm.~2]{ivanovs2018zooming}) and also $\beta^0=1/H$ by \cite[Cor.~1]{ivanovs2018zooming}. From Lemma \ref{lemma:q<beta} it follows that $\tau^0(q)=-q/\beta^0$ for $0<q<\beta^0$ and by Lemma \ref{lemma:q>beta} (iii) $\tau^0(q)=-1$ for $q>\beta^0$. By continuity of $\tau^0$ then $\tau^0(\beta^0)=-1$.
\end{enumerate}
\end{proof}

The scaling functions obtained in Theorem \ref{thm:tau0} are plotted in Figure \ref{fig:tau0}. If under the assumptions of Theorem \ref{thm:tau0} we put
\begin{equation*}
\alpha = \begin{cases}
2, & \ \text{ if } \hX \text{ is Brownian motion},\\
1, & \ \text{ if } \hX \text{ is linear drift},\\
\beta^0, & \ \text{ if } \hX \text{ is strictly stable process},\\
\end{cases}
\end{equation*}
then, when $\Pi\not \equiv 0$, $\tau^0$ can be written in a unified way for $q \in (0,\beta^\infty)$ as
\begin{equation*}
\tau^0(q)=\begin{cases}
-\frac{1}{\alpha} q, & \ \text{ if } 0<q\leq \alpha,\\
-1, & \ \text{ if } q>\alpha.
\end{cases}
\end{equation*}
Provided that the range of finite moments is large enough, the L\'evy process exhibits intermittency in small-time as soon as $\Pi\not \equiv 0$. Namely, we must have $\beta^\infty>\alpha$ to cover the non-linearity of $\tau^0$ so that
\begin{equation*}
\frac{\tau^0(q)}{q} = \begin{cases}
-\frac{1}{\alpha}, & \ \text{ if } 0<q\leq \alpha,\\
-\frac{1}{q}, & \ \text{ if } \alpha < q < \beta^\infty,
\end{cases}
\end{equation*}
is strictly increasing on $(\alpha,\beta^\infty)$. If $\beta^\infty<\alpha$, in particular, the variance of $X(1)$ is infinite, then $\tau^0$ is linear in the range of finite moments and there is no intermittency.

We note that Theorem \ref{thm:tau0} (and Lemmas \ref{lemma:q<beta} and \ref{lemma:q>beta}) establish the moment asymptotics for \textit{every} finite absolute moment of any L\'evy process satisfying \eqref{e:LPlimit}. The assumption \eqref{e:LPlimit} is very weak and is satisfied for almost every L\'evy process of practical interest. The most notable exceptions are driftless processes with no Gaussian component such that $\overline{\Pi}(x)=\Pi((x,\infty))+\Pi((-\infty,-x)) \in \RV^0(0)$. Particular examples of such processes are driftless compound Poisson process, driftless gamma and driftless variance gamma process. For more details and examples see \cite{bisewski2019zooming}.

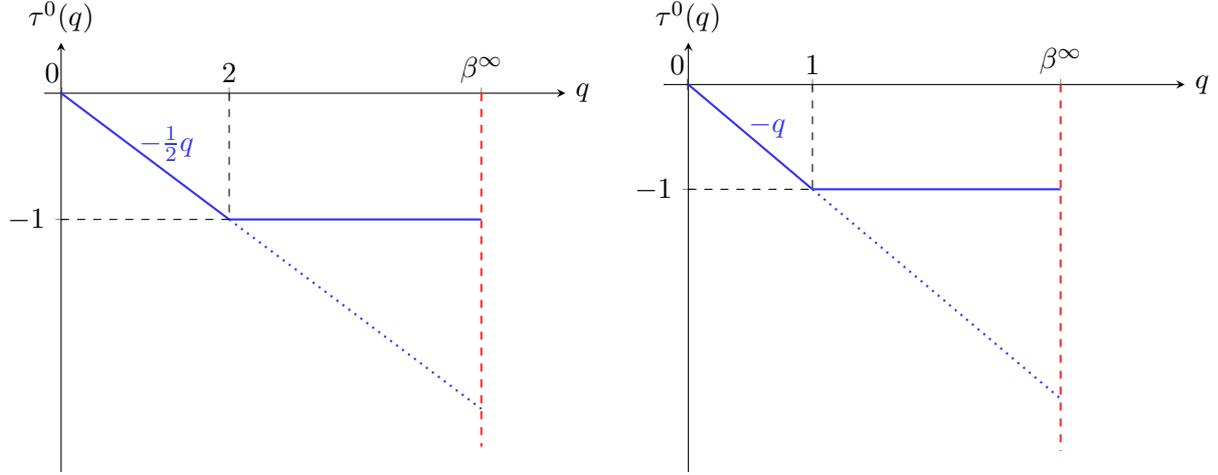
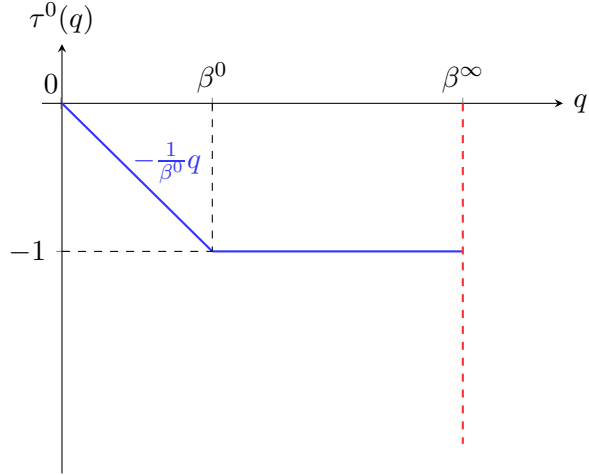
\begin{figure}[ht!]
\centering
\begin{subfigure}[b]{0.45\textwidth}
\begin{tikzpicture}[domain=0:6]
\begin{axis}[
axis lines=middle,
xlabel=$q$, xlabel style={at=(current axis.right of origin), anchor=west},
ylabel=$\tau^0(q)$, ylabel style={at=(current axis.above origin), anchor=south},
xtick={0.01,2,5},
xticklabel style={yshift=0.5ex, anchor=south},
xticklabels={$0\ \ $,$2$,$\beta^\infty$},
ytick={-1},
yticklabels={$-1$},
xmin=-0.2,
xmax=6,
ymin=-3,
ymax=0.4,
]
\addplot[thick,white!20!blue,domain=0:2]{-0.5*x} node [pos=0.4,right]{$-\frac{1}{2}q$};
\addplot[thick,white!20!blue,domain=2:5]{-1};
\addplot[dotted,white!20!blue,thick] coordinates {(2,-1) (5,-2.5)};
\addplot[dashed] coordinates {(2,-0) (2,-1)};
\addplot[dashed] coordinates {(0,-1) (2,-1)};
\addplot[thick,dashed,white!20!red] coordinates {(5,-0) (5,-2.8)};
\end{axis}
\end{tikzpicture}
\caption{Case (i) of Theorem \ref{thm:tau0} when $\Pi\not \equiv 0$ (solid) and $\Pi \equiv 0$ (dotted)}
\label{fig:tau0i}
\end{subfigure}
\hfill
\begin{subfigure}[b]{0.45\textwidth}
\begin{tikzpicture}[domain=0:6]
\begin{axis}[
axis lines=middle,
xlabel=$q$, xlabel style={at=(current axis.right of origin), anchor=west},
ylabel=$\tau^0(q)$, ylabel style={at=(current axis.above origin), anchor=south},
xtick={-0.01,1,3},
xticklabel style={yshift=0.5ex, anchor=south},
xticklabels={$0\ \ $,$1$,$\beta^\infty$},
ytick={-1},
yticklabels={$-1$},
xmin=-0.2,
xmax=4,
ymin=-3.7,
ymax=0.4,
]
\addplot[thick,white!20!blue,domain=0:1]{-x} node [pos=0.4,right]{$-q$};
\addplot[thick,white!20!blue,domain=1:3]{-1};
\addplot[dotted,white!20!blue,thick] coordinates {(1,-1) (3,-3)};
\addplot[dashed] coordinates {(1,-0) (1,-1)};
\addplot[dashed] coordinates {(0,-1) (1,-1)};
\addplot[thick,dashed,white!20!red] coordinates {(3,0) (3,-3.5)};
\end{axis}
\end{tikzpicture}
\caption{Case (ii) of Theorem \ref{thm:tau0} when $\Pi\not \equiv 0$ (solid) and $\Pi \equiv 0$ (dotted)}
\label{fig:tau0i}
\end{subfigure}
\hfill\\
\begin{subfigure}[b]{0.45\textwidth}
\begin{tikzpicture}[domain=0:6]
\begin{axis}[
axis lines=middle,
xlabel=$q$, xlabel style={at=(current axis.right of origin), anchor=west},
ylabel=$\tau^0(q)$, ylabel style={at=(current axis.above origin), anchor=south},
xtick={-0.01,1.5,4},
xticklabel style={yshift=0.5ex, anchor=south},
xticklabels={$0\ \ $,$\beta^0$,$\beta^\infty$},
ytick={-1},
yticklabels={$-1$},
xmin=-0.2,
xmax=5,
ymin=-2.5,
ymax=0.4,
]
\addplot[thick,white!20!blue,domain=0:1.5]{-x/1.5} node [pos=0.4,right]{$-\frac{1}{\beta^0}q$};
\addplot[thick,white!20!blue,domain=1.5:4]{-1};
\addplot[dashed] coordinates {(1.5,0) (1.5,-1)};
\addplot[dashed] coordinates {(0,-1) (1.5,-1)};
\addplot[thick,dashed,white!20!red] coordinates {(4,0) (4,-2.3)};
\end{axis}
\end{tikzpicture}
\caption{Case (iii) of Theorem \ref{thm:tau0}}
\label{fig:tau0i}
\end{subfigure}
\caption{The scaling function \eqref{deftau0} of L\'evy process $X$ in Theorem \ref{thm:tau0}}
\label{fig:tau0}
\end{figure}

\section{A note on the multifractal formalism}\label{sec5}
The multifractal formalism first appeared in the field of turbulence theory (\cite{frisch1985fully,frisch1995turbulence}) and relates local and global scaling properties of an object. For stochastic processes one may describe local regularity by roughness of the process sample paths measured by the pointwise H\"older exponents. Let $t_0 \in [0,\infty)$ and $\gamma>0$. We say that a function $f: [0,\infty) \to \mathbb{R}$ is $C^{\gamma}(t_0)$ if there exists a constant $C>0$ and polynomial $P_{t_0}$ of degree at most $\lfloor \gamma \rfloor$ such that for all $t$ in some neighborhood of $t_0$ it holds that $|f(t)-P_{t_0}(t)| \leq C |t - t_0|^{\gamma}$. If the Taylor polynomial of this degree exists, then $P_{t_0}$ is that Taylor polynomial. Thus, if $P_{t_0}$ is constant, then $P_{t_0} \equiv f(t_0)$. This happens, in particular, when $\gamma<1$ (see \cite{riedi2003multifractal}). 
It is clear that if $f \in C^{\gamma}(t_0)$, then $f \in C^{\gamma'}(t_0)$ for each $0<\gamma'<\gamma$. A pointwise H\"older exponent of the function $f$ at $t_0$ is $H(t_0)= \sup \left\{ \gamma : f \in C^{\gamma}(t_0) \right\}$.

In the multifractal analysis one analyzes the sets $S_h=\{ t : H(t)=h \}$ consisting of the points in the domain where $f$ has the H\"older exponent of value $h$. In the interesting examples, sets $S_h$ are fractal and one typically uses Hausdorff dimension $\dim_H S_h$ of these sets to measure their size. The mapping $h \mapsto d(h)=\dim_H S_h$ is called the spectrum of singularities (also multifractal or Hausdorff spectrum) with the convention that $\dim_H \emptyset=-\infty$. The set of $h$ such that $d(h)\neq - \infty$ is referred to as the support of the spectrum. A function $f$ is said to be multifractal if the support of its spectrum is non-trivial, in the sense that it is not a one point set. 

When considered for a stochastic process, H\"older exponents are random variables and $S_h$ random sets. However, in many cases the spectrum is deterministic, that is, for almost every sample path spectrum is the same. Moreover, the spectrum is usually homogeneous, in the sense it is the same when considered over any nonempty subset $A\subset [0,\infty)$. 

It is well-known that the sample paths of Brownian motion are monofractal, i.e.~$d(1/2)=1$ and $d(h)=-\infty$ for $h\neq 1/2$. However, other L\'evy processes typically have multifractal paths \cite{jaffard1999levy,balanca2013}. Namely, if there is no Gaussian component and $\beta^0>0$, then the spectrum of singularities is a.s.
\begin{equation}\label{specLP}
d(h) =
\begin{cases}
\beta^0 h, & \text{if } \ h \in [0, 1/\beta^0],\\
-\infty, & \text{if } \ h > 1/\beta^0.
\end{cases}
\end{equation}

Multifractal formalism states that the singularity spectrum of the function or a process can be obtained from some function $\zeta$ by computing
\begin{equation}\label{formalism}
d(h)= \inf_q \left( hq - \zeta(q) +1\right).
\end{equation}
Here $\zeta$ quantifies some global property of the process and its definition depends on the context. In \cite{frisch1985fully}, $\zeta$ is defined as the asymptotic scaling of moments of increments which would correspond to taking $\zeta(q) = \lim_{t\to 0} \frac{\log \E |X(t)|^q}{\log t}$. In the theory of multifractals processes such $\zeta$ is referred to as the scaling function and describes the asymptotic scale invariance of moments (see e.g.~\cite{muzy2002multifractal,muzy2013random,allez2013lognormal,rhodes2014levy}). In scenario (iii) of Theorem \ref{thm:tau0} we would have that
\begin{equation*}
\zeta(q)=\begin{cases}
\frac{1}{\beta^0} q, & \ \text{ if } 0<q\leq \beta^0,\\
1, & \ \text{ if } \beta^0 < q < \beta^\infty.
\end{cases}
\end{equation*}
It is easy to check then that the multifractal formalism \eqref{formalism} holds for L\'evy processes, at least when there is no Brownian component. It remains an open problem whether it is possible to derive \eqref{specLP} directly from Theorem \ref{thm:tau0}.

\bigskip

\bibliographystyle{agsm}
\bibliography{References}

\end{document}